\documentclass[leqno]{article}
\usepackage{amssymb,amsmath,amsfonts,amscd,verbatim}
\usepackage{color,graphicx,caption,subcaption}
\usepackage{algorithm,algorithmic}
\usepackage{upgreek,comment,url,comment}
\usepackage{float,units}
\usepackage{hyperref}
\usepackage{xcolor}
\usepackage[margin=1in]{geometry}

\input{./Definitions}

\usepackage[
backend=biber,
style=numeric,
]{biblatex}

\title{A Note on the Stability of the Sherman-Morrison-Woodbury Formula}
\author{
Linkai Ma\thanks{Department of Computer Science, Purdue University, West Lafayette, IN 47907, USA, \href{mailto:ma856@purdue.edu}{ma856@purdue.edu}}\and
Christos Boutsikas\thanks{Department of Computer Science, Purdue University, West Lafayette, IN 47907, USA, \href{mailto:cboutsik@purdue.edu}{cboutsik@purdue.edu}}\and
Mehrdad Ghadiri\thanks{Massachusetts Institute of Technology
, Cambridge, MA 02139, USA, \href{mailto:mehrdadg@mit.edu}{mehrdadg@mit.edu}}\and
Petros Drineas\thanks{Department of Computer Science, Purdue University, West Lafayette, IN 47907, USA, \href{mailto:pdrineas@purdue.edu}{pdrineas@purdue.edu}}
}

\begin{document}
\maketitle

\begin{abstract}
We study the numerical stability of the Sherman–Morrison–Woodbury (SMW) identity. Let $\Bb = \Ab + \Ub \Vb^T$ and assume $\Ub$ and $\Vb$ both have full-column rank. We explore error bounds for the SMW identity when we are only able to compute approximate inverses. For both forward and backward errors, we present upper bounds as a function of the two-norm error of the approximate inverses. We verify with numerical experiments that, in certain cases, our bounds accurately capture the behavior of the errors.
\end{abstract}

\section{Introduction}

The Sherman-Morrison-Woodbury (SMW) formula is a fundamental tool in numerical linear algebra, providing an efficient approach for updating the inverse of a matrix under low-rank modifications. It is widely used in computational and applied mathematics, as well as in statistics~\cite{guttel2024sherman,nino2015efficient}, economics~\cite{riddell1975recursive}, mechanical engineering~\cite{wu2020model}, etc. For an extensive review of the history and applications of SMW see~\cite{hager1989updating}.

The numerical stability of the SMW formula, especially in high-dimensional or ill-conditioned settings, remains an important concern and an area of ongoing research. Two important contributions in this line of research were the works of Yip~\cite{yip1986note} and, more recently, Ghadiri~\textit{et al.}~\cite{ghadiri2023bit} (see Section~\ref{prior_work} for detailed discussion of prior work). Our point of departure in this paper, following the lines of~\cite{ghadiri2023bit}, is the case where one can only compute approximate inverses while applying the SMW formula. Recall that, for an invertible matrix $\Ab \in \mathbb{R}^{n \times n}$ and matrices $\Ub,\Vb \in \mathbb{R}^{n \times k}$, the SMW formula expresses the inverse of $\Bb = \Ab + \Ub \Vb^T$ as an update to $\Ab^{-1}$:
\begin{equation*}
\Bb^{-1} = \Ab^{-1} - \Ab^{-1} \Ub (\Ib + \Vb^T \Ab^{-1} \Ub)^{-1} \Vb^T \Ab^{-1}.
\end{equation*}
Importantly, one needs to assume that the so-called \textit{capacitance matrix} $\Ib + \Vb^T \Ab^{-1} \Ub$ is invertible. We consider the setting where $\Ab^{-1}$ and/or the capacitance matrix are only inverted approximately, i.e., $\Ab^{-1}$ is approximated by $\Abtil^{-1}$ and the inverse of the capacitance matrix is approximated by $\Zb^{-1}$. Then, our approximate inverse of $\Bb$ becomes:
\begin{equation}
    \widetilde{\Bb} ^{-1} = \Abtil^{-1} - \Abtil^{-1} \Ub \Zb^{-1} \Vb^T \Abtil^{-1}. \label{approx_inv_formula}
\end{equation}
Our work provides \textit{forward} and \textit{backward} error bounds for the update formula of eqn.~(\ref{approx_inv_formula}). Specifically, we bound:
\begin{itemize}
    \item Forward Error (Theorem~\ref{forward_error_thm}): Is $\widetilde{\Bb} ^{-1}$ close to $\Bb^{-1}$?
    \item Backward Error (Theorem~\ref{backward_thm}): Is $\widetilde{\Bb}$ close to $\Bb$? 
\end{itemize}
In both cases, we seek to bound the two-norms of the error matrices $\widetilde{\Bb} ^{-1}-\Bb^{-1}$ and $\widetilde{\Bb}-\Bb$ as a function of the approximation errors in computing the inverses of $\Ab$ and the capacitance matrix. Specifically, let $\epsilon_1$ and $\epsilon_2$ denote the following approximation errors:
\begin{flalign}
\epsilon_1 &= \| \Abtil^{-1} - \Ab^{-1} \|_2,\label{eqn:assumption1}   \\
\epsilon_2 &= \| \Zb^{-1} - \inv{\Ib + \Vb^T \Abtil^{-1} \Ub} \|_2.\label{eqn:assumption2}
\end{flalign}
Our precise forward and backward error bounds are technical and involved and we defer their presentation to Section~\ref{sec_proof}. Here, we focus on simplified versions of our bounds. Let $\sigma_{\min}(\Ab)$ denote the smallest singular value of $\Ab$.
Consider the case where we have a ``small'' update of the form $\Bb = \Ab + \Ub \Vb^T$, with $\twonorm{\Ub} \twonorm{\Vb} \leq .5\cdot \sigma_{\min}(\Ab)$. 
%
%
Then (see Corollary~\ref{remark:forward} for details), our forward error bound reduces to:
\begin{equation}
    \twonorm{\widetilde{\Bb}^{-1}-\Bb^{-1}} \leq 2\epsilon_2 \twonorm{\Ab^{-1}} + 12\epsilon_1.
    \label{simple_forward_bound}
\end{equation}
Under the same assumptions (see Corollary~\ref{remark:backward} for details), our backward error bound simplifies to:
\begin{equation}
    \twonorm{\widetilde{\Bb}-\Bb} \leq 2 \epsilon_1 \twonorm{\Ab}^2 + 8\epsilon_2.
    \label{simple_backward_bound}
\end{equation}
The two bounds presented above reveal an intriguing dichotomy, which, to the best of our knowledge, is a novel observation. For small updates, the forward error bound for the SMW formula is primarily influenced by the approximation error in the inversion of the capacitance matrix and scales linearly with this error and the two-norm of $\Ab^{-1}$. Intuitively, the forward error is sensitive to both the proximity of $\Ab$ to singularity, as measured by the inverse of its smallest singular value (recall that $\|\Ab^{-1}\|_2 = \nicefrac{1}{\sigma_{\min}(\Ab)}$), and the magnitude of the capacitance matrix inversion error. Conversely, the backward error for the SMW formula is dominated by the approximation error in the inversion of $\Ab$ and scales linearly with this error and the (square of) the two-norm of $\Ab$. This dichotomy, validated by our numerical experiments in Section \ref{sec_numerical_experiment}, provides practical guidance: practitioners should prioritize the numerical accuracy of inverting either $\Ab$ or the capacitance matrix based on the desired type of error guarantee, forward or backward. 

Another interesting case is when the capacitance matrix is not well-behaved. Let $\|\Ub\|_2\|\Vb\|_2 = \lambda$ and let $\epsilon_1 = \epsilon_2 =\epsilon$.
Suppose $\twonorm{(\Ib + \Vb^T \Ab^{-1} \Ub)^{-1}}\leq \alpha$ is large (see Corollary~\ref{alpha large} for details). Then, our forward error bound is dominated by the term
\begin{equation}
    \epsilon \alpha^2\lambda^2 \twonorm{\Ab^{-1}}^2.
    \label{simple_forward_bound_2}
\end{equation}
Now suppose $\twonorm{\Ib + \Vb^T \Ab^{-1} \Ub}\leq \beta$ is large (see Corollary~\ref{beta large} for details). Then, our backward error bound will be dominated by the term
\begin{equation}
    \epsilon \lambda \beta^2.
    \label{simple_backward_bound_2}
\end{equation}
Our experiments in Section~\ref{vary capacitance} verify that our error bounds successfully capture the behavior of the actual errors. The forward error grows linearly with the (square of) the inverse of the smallest singular value of the capacitance matrix, while the backward error is affected by the (square of) the largest singular value of the capacitance matrix.


The paper is organized as follows: Section~\ref{sec_background} introduces notation and reviews prior work. Section~\ref{sec_proof} states our forward and backward error bounds and their proofs. Section \ref{sec_numerical_experiment} presents numerical experiments to verify our theory and Section \ref{sec_future} discusses potential directions for future work.

\section{Background} \label{sec_background}

\subsection{Notation}
We denote scalars using lowercase letters (i.e., $x$, $y$). We denote vectors using bold lowercase letters (i.e., $\xb$, $\yb$). We denote matrices using bold uppercase letters (i.e., $\Ab$, $\Bb$).
We use $\bold{e_i}$ to denote the $i$-th canonical vector and $\Ib$ to denote the identity matrix where its dimensions will be clear from the context. We use $\Ab^{\dagger}$ to denote the Moore–Penrose pseudoinverse of matrix $\Ab$; $\|\Ab\|_2$ to denote the spectral norm of a matrix; and $\kappa(\Ab) = \|\Ab\|_2 \| \Ab^{\dagger} \|_2 $ to denote its condition number. 

%

\subsection{Prior Work}\label{prior_work} The stability of the SMW formula has been a subject of interest in numerical linear algebra starting with work by G.~W.~Stewart~\cite{stewart1974modifying}. His work analyzed the stability of the formula for rank-one updates in the context of modifying pivot elements in Gaussian elimination. Later, Yip~\cite{yip1986note} extended the stability analysis to full-rank updates and presented bounds on the condition number of the capacitance matrix. Specifically, given invertible square matrices $\Ab$ and $\Bb$ 
%
%
and rank-$k$ update matrices $\Ub$ and $\Vb$, the condition number of the capacitance matrix is bounded by
\begin{equation} \label{yip_1}
    \kappa \left( \Ib - \Vb^{T} \Ab^{-1} \Ub \right) \leq \min \{ \kappa^2(\Ub), \kappa^2\left( \Vb^{T} \right) \} \kappa(\Ab) \kappa(\Bb).
\end{equation}
Moreover, under the assumption that the rank-$k$ matrices $\Ub, \, \Vb$ have a more specific structure - in particular, that $\Ub \Vb^{T}$ consists of exactly $n-k$ zero rows (columns) and the remaining $k$ rows (columns) are linearly independent,~\cite{yip1986note} simplified the bound of eqn.~(\ref{yip_1}) to
\begin{equation} \label{yip_2}
    \kappa \left( \Ib - \Vb^{T} \Ab^{-1} \Ub \right) \leq \kappa(\Ab) \kappa(\Bb).
\end{equation}
The motivation underlying the study of the conditioning of the capacitance matrix is that it may affect the stability of SMW when solving a linear system using the perturbed matrix $\Bb$, even if both $\Ab$ and $\Bb$ are well-conditioned. 

A more recent line of research on the stability of SMW was motivated by analyzing bit complexity in optimization problems. In this setting, inverse maintenance problems are closely connected with the stability of the SMW formula. In this context, Ghadiri \textit{et al.}~\cite{ghadiri2023bit} analyzed the backward stability of the SMW formula and presented Frobenius norm bounds. More precisely, given invertible matrices $\Ab \in \R^{n \times n}$, $\tilde{\Ab} \in \R^{n \times n}$, $\Cb \in \R^{k \times k}$, and assuming that the matrices $\Ub,\Vb \in \R^{n \times k}$, and the matrix $\Cb \in \mathbb{R}^{k \times k}$ satisfy the condition that $\Ab + \Ub \Cb \Vb^{T}$ is invertible, Lemma I.11 in~\cite{ghadiri2023bit} states that
\begin{equation} \label{Ghadiri_main_result}
    \left\Vert \left( \tilde{\Ab}^{-1} - \tilde{\Ab}^{-1} \Ub \Zb^{-1} \Vb^{T} \tilde{\Ab}^{-1}  \right)^{-1} - \left( \Ab + \Ub \Cb \Vb^{T} \right) \right\Vert_{F} \leq 512 \rho ^{26} \epsilon_{2} + \epsilon_{1}.
\end{equation}
Unlike our setting,~\cite{ghadiri2023bit} assumes that $\Ab$ instead of $\Ab^{-1}$ is approximated by $\tilde{\Ab}$:
\begin{equation*}
    \left\Vert \tilde{\Ab} -\Ab \right\Vert_{F} \leq \epsilon_{1}<1.
\end{equation*}
Similar to our setting, $\Zb^{-1}$ is the approximate inverse of the capacitance matrix and satisfies
\begin{equation*}
    \left\Vert \Zb^{-1} - \left( \Cb^{-1}  + \Vb^{T} \tilde{\Ab}^{-1} \Ub \right)^{-1}  \right\Vert_{F} \leq \epsilon_{2} < 1.
\end{equation*}
We also note that the bound of eqn.~(\ref{Ghadiri_main_result}) assumes that 
\begin{align*}
    \max  \bigg\{ 
    &\| \Ab \|_{F} , \, \| \Ab^{-1} \|_{F}, \| \Cb \|_{F} , \, \| \Cb^{-1} \|_{F}, \, \|\Ub \|_{F}, \, \| \Vb \|_{F}, \\ 
    &\left\Vert \Ab + \Ub \Vb^{T}\right\Vert_{F}, \, \left\Vert (\Ab + \Ub \Vb^{T} )^{-1}\right\Vert_{F} \bigg\} \leq \rho.
\end{align*}
\cite{ghadiri2023bit} also assumes that $\rho > n + k$, which forces the above error bound to scale as a function of at least $(n+k)^{26}$. This assumption in~\cite{ghadiri2023bit} arises due to the use of the Frobenius norm, but our work does not require it.

\subsection{Improving the bound of~\cite{ghadiri2023bit} for the two-norm case} The bound of eqn.~(\ref{Ghadiri_main_result}) applies to the generalized SMW identity, also known as the Woodbury matrix identity~\cite{woodbury1950inverting}. For the special case where $\Cb = \Ib$, we derive a novel bound for the \textit{two-norm} backward error, significantly improving the dependency on $\rho$ \textit{and} removing the assumption that $\rho$ needs to be strictly larger than $n + k$.
Precisely, assume 
\begin{align*}
    \max  \bigg\{ 
    \| \Ab \|_{2} , \, \| \Ab^{-1} \|_{2}, \, \|\Ub \|_{2}, \, \| \Vb \|_{2},
    \left\Vert \Ab + \Ub \Vb^{T}\right\Vert_{2}, \, \left\Vert (\Ab + \Ub \Vb^{T} )^{-1}\right\Vert_{2} \bigg\} \leq \rho.
\end{align*}
Also assume
\begin{align*}
    \twonorm{\tilde{\Ab} - \Ab} \leq \epsilon_{1} < 1\quad \mbox{and} \quad \twonorm{\Zb^{-1} - \left( \Ib +  \Vb^{T} \tilde{\Ab}^{-1} \Ub \right)^{-1}} \leq \epsilon_{2},
\end{align*}
with\footnote{Actually, $\epsilon_2\leq \nicefrac{\rho^3}{8 (1+\rho)^4 (\rho^3+1)^2}$ suffices.} $\epsilon_{2} \leq \nicefrac{1}{512\rho^7}$. Then, 
\begin{equation} \label{Ghadiri_main_result_norm_2}
    \left\Vert \left( \tilde{\Ab}^{-1} - \tilde{\Ab}^{-1} \Ub \Zb^{-1} \Vb^{T} \tilde{\Ab}^{-1}  \right)^{-1} - \left( \Ab + \Ub  \Vb^{T} \right) \right\Vert_{F} \leq 512 \epsilon_2 \rho^{14} + \epsilon_{1}.
\end{equation}
See Appendix~\ref{Ghadiri_main_result_norm_2_proof} for the proof.

\subsection{Two lemmas from prior work}
We will make frequent use of Lemma~\ref{first_lemma}, which allows us to bound the two-norm difference between the inverses of two ``close'' matrices. Lemma \ref{first_lemma} is adapted from Lemma II.1 in~\cite{ghadiri2023bit}; we note that~\cite{ghadiri2023bit} proved a Frobenius norm bound, while we are interested in two-norm bounds.
\begin{lemma}
 Let  $\Mb, \Nb$ be $n\times n$ invertible matrices and let $\|\Nb^{-1}\|_2 \leq \rho$ for some $\rho >1$. Assume $\|\Mb-\Nb\|_2 \leq \epsilon \leq \nicefrac{1}{2\rho}$, then $\|\Mb^{-1}\|_2 \leq 2 \rho$, and
\begin{equation*}
    \|\Mb^{-1}-\Nb^{-1}\|_2 \leq 2 \rho^2 \epsilon.
\end{equation*}
\label{first_lemma}
\end{lemma}
The proof is identical to the proof of Lemma II.1 in~\cite{ghadiri2023bit}, using the two-norm instead of the Frobenius norm. For completeness, we reproduce the proof in Appendix~\ref{first_lemma_proof}.
\begin{remark}
It is easy to see that Lemma~\ref{first_lemma} is essentially tight. First of all, in a rather degenerate case, if we set $\Mb = \nicefrac{c}{2}\, \Ib$, $\Nb = c\Ib$ for some constant $c$. Let $\epsilon = \nicefrac{c}{2} $, $\rho = \nicefrac{1}{c}$, then, 
$$\twonorm{\Mb^{-1}-\Nb^{-1}} = \twonorm{(\nicefrac{2}{c})\Ib - (\nicefrac{1}{c})\Ib} = \nicefrac{1}{c}= 2 \rho^2 \epsilon.$$ 
Therefore, Lemma \ref{first_lemma} is tight. More generally, assume that $\Nb$ has full rank and let $\Nb = \sum_{i=1}^n \sigma_i(\Nb) \ub_i \vvb_i^T$ be the SVD of $\Nb$, where $\ub_i$ and $\vvb_i$ are the left and right singular vectors of $\Nb$, respectively, and $\sigma_i(\Nb)>0$. 
Let $\Mb = \Nb + \frac{1}{2} \sigma_n \ub_n \vvb_n^T$.
Then, $\|\Nb^{-1}\|_2 = \frac{1}{\sigma_n}$, and $\|\Mb - \Nb\|_2 = \frac{1}{2} \sigma_n$. Let $\rho = \frac{1}{\sigma_n}$ and $\epsilon = \frac{1}{2} \sigma_n $. It follows that
\begin{equation}
    \|\Mb^{-1} - \Nb^{-1}\|_2= \frac{1}{3} \sigma_n ^{-1} = \frac{1}{3} \left( 2 \rho^2 \epsilon \right).
\end{equation}
Therefore, Lemma~\ref{first_lemma} is tight for general matrices as well, perhaps up to a small constant factor.
\end{remark}

\begin{lemma}\label{lem:lemma_NL_1}
    Assuming $\Ub \text{ and }\Vb$ have full column rank; $\Ab \text{ and }\Bb$ are invertible; $\Bb = \Ab + \Ub \Vb^T$, the following identities hold:
    \begin{align}
       & \Ib + \Vb^T \Ab^{-1} \Ub = \Vb^T \Ab^{-1} \Bb (\Vb^T)^\dagger,\ \mbox{and} \label{identity1}\\
       & \left( \Ib + \Vb^T \Ab^{-1} \Ub \right)^{-1} = \Vb^T \Bb^{-1}\Ab (\Vb^T)^\dagger.\label{identity2}
    \end{align}
\end{lemma}
Lemma~\ref{lem:lemma_NL_1} is adapted from eqns.~(6) and~(7) in~\cite{yip1986note}. (See Appendix~\ref{appendix_proof_1} for proofs.)
\section{Our Bounds}\label{sec_proof}

We now present in detail both our forward and backward error bounds. 

\subsection{Forward Error Bound}

We first discuss necessary assumptions for our forward error bound; next, we present and discuss our bound and the special case mentioned in the introduction; and finally we present its proof.

\subsubsection{Assumptions}\label{assumptions}

We assume $\Ub$ and $\Vb$ have full column rank and use $\Ub^\dagger, \Vb^\dagger$ to denote their Moore–Penrose inverses. Let 
%
$$\lambda = \twonorm{\Ub} \twonorm{\Vb}.$$ 
Also recall the definitions of $\epsilon_1$ and $\epsilon_2$ in eqns.~(\ref{eqn:assumption1}) and~(\ref{eqn:assumption2}), since our (forward) error bound will depend on the parameters $\epsilon_1$ and $\epsilon_2$, which characterize the quality of the approximate inverses that are used in the SMW identity.

\subsubsection{The bound}\label{sxn:factualbound}
We are now ready to state our forward error bound for the SMW formula.
\begin{theorem}
If $\twonorm{\inv{\Ib + \Vb^T \invnop{\Ab}\Ub}}\leq \alpha$ and
    \begin{equation}
        \epsilon_1 < \frac{1}{2 \,\lambda\,\alpha},
        \label{forward_bound_assumption}
    \end{equation}
    then
    \begin{flalign}
    \twonorm{\Bb^{-1} - \left(\Abtil^{-1} - \Abtil^{-1} \Ub \Zb^{-1} \Vb^T \Abtil^{-1}\right)} & \leq \epsilon_1 + \epsilon_1\lambda \alpha 
    \left( 2 \twonorm{\Ab^{-1}} + \epsilon_1 \right) \nonumber\\
    &+ \lambda(\twonorm{\Ab^{-1}} + \epsilon_1)^2 \left( \epsilon_2 + 2\epsilon_1 \lambda \,\alpha^2 \right).\label{forward_bound}
\end{flalign}
\label{forward_error_thm}
\end{theorem}
The first two terms in the above bound depend only on the approximation error $\epsilon_1$, while the third term depends on both $\epsilon_1$ and $\epsilon_2$. The above bound is, not unexpectedly, complicated, and we focus on the special cases discussed in our introduction to help the reader parse it. It is worth noting that if both $\epsilon_1$ and $\epsilon_2$ are equal to zero, then the bound is zero as well. 

\begin{corollary}\label{remark:forward}
Assume that $\epsilon_1$, $\epsilon_2$  and $\sigma_{\min}(\Ab)$ are all upper bounded by one and also assume that $\epsilon_1$ is sufficiently small to satisfy the assumption of Theorem~\ref{forward_error_thm}. Additionally, assume that we have a small update, i.e., $\lambda \leq .5\cdot \sigma_{\min}(\Ab)$. Then, the right-hand side of the bound of eqn.~(\ref{forward_bound}) simplifies to:
\begin{equation*}
   \twonorm{\Bb^{-1} - \left(\Abtil^{-1} - \Abtil^{-1} \Ub \Zb^{-1} \Vb^T \Abtil^{-1}\right)} \leq  2\epsilon_2 \twonorm{\Ab^{-1}} +12\epsilon_1.
\end{equation*}
\end{corollary}

\begin{proof}
   Using $\lambda \leq .5\cdot \sigma_{\min}(\Ab)$, we get $\lambda \twonorm{\Ab^{-1}}\leq .5$. Thus,
   \begin{align*}
    \twonorm{\left(\Ib+\Vb^T \Ab^{-1} \Ub\right) - \Ib }&= \twonorm{\Vb^T \Ab^{-1} \Ub }\leq \lambda \twonorm{\Ab^{-1}} \leq .5.
   \end{align*} 
    The above inequality suggests that $\Ib + \Vb^T \Ab^{-1} \Ub$ is close to $\Ib$. We apply Lemma~\ref{first_lemma} to get
    \begin{equation*}
        \alpha = \twonorm{\inv{\Ib + \Vb^T \Ab^{-1} \Ub}} \leq 2.
    \end{equation*}
    We now switch our attention to the right-hand side of the bound of eqn.~(\ref{forward_bound}).
    The first term is simply $\epsilon_1$. For the second term, recall that $\epsilon_1 \leq 1$ and  $\sigma_{\min}(\Ab) \leq 1$ to get $\epsilon_1 \leq \twonorm{\Ab^{-1}}$. Then, the second term can be bounded as follows:
    \begin{align*}
        \epsilon_1 \lambda \left( 2 \twonorm{\Ab^{-1}} + \epsilon_1 \right) \alpha &\leq 3 \epsilon_1 \lambda \twonorm{\Ab^{-1}}\,\alpha
        \\
        & \leq 6 \epsilon_1 \lambda \twonorm{\Ab^{-1}}\\
        &\leq 3\epsilon_1.
    \end{align*}
    Finally, the third term can be bounded as follows:   
    \begin{align*}
        \lambda(\twonorm{\Ab^{-1}} + \epsilon_1)^2 \left( \epsilon_2 + 2\epsilon_1 \lambda \,\alpha^2 \right)
        &\leq 4 \lambda \twonorm{\Ab^{-1}}^2 \left( \epsilon_2 + 2\epsilon_1 \lambda \,\alpha^2\right)\\ 
        &\leq 4 \lambda \twonorm{\Ab^{-1}}^2 \left( \epsilon_2 + 8\epsilon_1 \lambda  \right)\\
        &= 4 \lambda \twonorm{\Ab^{-1}}^2 \epsilon_2 + 32 \lambda^2 \twonorm{\Ab^{-1}}^2 \epsilon_1\\
        &\leq 2\epsilon_2 \twonorm{\Ab^{-1}} + 8\epsilon_1. 
    \end{align*}
    The first inequality follows using $\epsilon_1\leq \|\Ab^{-1}\|_2$; the second inequality follows from our bound for $\|\inv{\Ib + \Vb^T \invnop{\Ab} \Ub}\|_2$; the last inequality follows from $\lambda \|\Ab^{-1}\|_2 \leq .5$.
\end{proof}


\begin{corollary}\label{alpha large}
    Assume that $\epsilon_1 = \epsilon_2 = \epsilon$ and $\sigma_{\min}(\Ab)$ are all upper bounded by one and also assume that $\epsilon_1$ is sufficiently small to satisfy the assumption of Theorem~\ref{forward_error_thm}. Additionally, assume that $\alpha$ is sufficiently large,i.e.,$\alpha \geq \max(\lambda^{-1},1)$. Then the right-hand side of the bound of eqn.~(\ref{forward_bound}) simplifies to:
    \begin{equation*}
        \twonorm{\Bb^{-1} - \left(\Abtil^{-1} - \Abtil^{-1} \Ub \Zb^{-1} \Vb^T \Abtil^{-1}\right)} \leq 16\epsilon \lambda^2 \twonorm{\Ab^{-1}}^2 \alpha^2 
    \end{equation*}
\end{corollary}
\begin{proof}
    
    Using $\alpha \geq \lambda^{-1}$, we get $\lambda^2 \alpha^2 \geq 1$. We assumed $\sigma_{\min}(\Ab)\leq 1$ and $\twonorm{\Ab^{-1}}\geq1$; therefore, the first term in the bound of Theorem~\ref{forward_error_thm} can be bounded as follows:
    \begin{equation*}
        \epsilon_1 \leq \epsilon \lambda^2 \twonorm{\Ab^{-1}}^2 \alpha^2.
    \end{equation*}
    For the second term, using $\epsilon_1 \leq 1$, we get $\epsilon_1 \leq \twonorm{\Ab^{-1}}$. Then the second term can be bounded similarly:
    \begin{align*}
        \epsilon_1 \lambda \alpha (2\twonorm{\Ab^{-1}}+\epsilon_1) &\leq 3 \epsilon_1 \lambda \alpha \twonorm{\Ab^{-1}}\\
        &\leq 3\epsilon \lambda^2 \alpha^2 \twonorm{\Ab^{-1}}^2.
    \end{align*}
    Now we switch our attention to the last term on the right hand side of eqn.~(\ref{forward_bound}). Using $\epsilon_1 \leq \twonorm{\Ab^{-1}}$, we get:
    \begin{align*}
        \lambda (\twonorm{\Ab^{-1}} + \epsilon_1)^2 \left( \epsilon_2 + 2\epsilon_1 \lambda \,\alpha^2 \right) &\leq 4 \lambda \twonorm{\Ab^{-1}}^2 \epsilon +8 \epsilon \lambda^2 \twonorm{\Ab^{-1}}^2  \alpha^2 .
    \end{align*}
    Since $\nicefrac{1}{\sqrt{\lambda}}\leq \max(\lambda^{-1},1)$, it follows that $\alpha \geq \nicefrac{1}{\sqrt{\lambda}}$. Then,
    \begin{equation*}
        4 \lambda \twonorm{\Ab^{-1}}^2 \epsilon \leq 4  \epsilon \lambda^2 \twonorm{\Ab^{-1}}^2 \alpha^2.
    \end{equation*}
    Finally, adding all the terms together, we get:
    \begin{equation*}
       \twonorm{\Bb^{-1} - \left(\Abtil^{-1} - \Abtil^{-1} \Ub \Zb^{-1} \Vb^T \Abtil^{-1}\right)} \leq 16\epsilon \lambda^2 \twonorm{\Ab^{-1}}^2 \alpha^2.
    \end{equation*}
    
\end{proof}
We also provide an alternative formulation for Theorem~\ref{forward_error_thm}.
\begin{corollary}\label{forward old bound}
    Under the assumptions of Theorem~\ref{forward_error_thm},
    \begin{flalign*}
        \twonorm{\Bb^{-1}-\left(\Abtil^{-1} - \Abtil^{-1} \Ub \Zb^{-1} \Vb^T \Abtil^{-1}\right)} &\leq \epsilon_1+ \epsilon_1\lambda \kappa(\Vb)\twonorm{\Bb^{-1}\Ab}
    \left( 2 \twonorm{\Ab^{-1}} + \epsilon_1 \right) \nonumber\\
    &+ \lambda(\twonorm{\Ab^{-1}} + \epsilon_1)^2 \left( \epsilon_2 + 2\epsilon_1 \lambda \,(\kappa(\Vb)\twonorm{\Bb^{-1}\Ab})^2 \right).
    \end{flalign*}
\end{corollary}

\begin{proof}
    Using eqn.~(\ref{identity2}), we can bound $\alpha$ as follows:
    \begin{align*}
        \alpha &= \twonorm{\inv{\Ib + \Vb^T \Ab^{-1} \Ub}}\\
        &=\twonorm{\Vb^T \Bb^{-1}\Ab (\Vb^T)^{\dagger}}\\
        &\leq \twonorm{\Vb^T} \twonorm{\Bb^{-1}\Ab} \twonorm{(\Vb^T)^{\dagger}} = \kappa(\Vb) \twonorm{\Bb^{-1}\Ab}.
    \end{align*}
    Substituting back into eqn.~(\ref{forward_bound}) yields the result.
\end{proof}

\subsubsection{Proof of Theorem~\ref{forward_error_thm}}
We now proceed with the proof of Theorem~\ref{forward_error_thm}. 
From Woodbury's identity,
\begin{equation*}
    \Bb^{-1} = \Ab^{-1} - \Ab^{-1} \Ub (\Ib + \Vb^T \Ab^{-1} \Ub)^{-1} \Vb^T \Ab^{-1}.
\end{equation*}
Let $\Eb_1=\Abtil^{-1}-\Ab^{-1}$ and $\Eb_2 = \Zb^{-1}- \left( \Ib + \Vb^T \Abtil^{-1} \Ub \right)^{-1}$. We break up $\widetilde{\Bb}^{-1}-\Bb^{-1}$ into five parts:
\begin{align}
       \widetilde{\Bb}^{-1}-\Bb^{-1} &= \left(\Abtil^{-1}- \Abtil^{-1}\Ub \Zb^{-1} \Vb^T \Abtil^{-1} \right) - \Bb^{-1} \nonumber\\
       &= \Eb_1 - \underbrace{\left( \Abtil^{-1}\Ub \Zb^{-1} \Vb^T \Abtil^{-1} - \Abtil^{-1} \Ub \left(\Ib+\Vb^T\Abtil^{-1} \Ub \right)^{-1} \Vb^T \Abtil^{-1}\right)}_{\Tb_1}\nonumber\\
       & - \underbrace{\left( \Abtil^{-1} \Ub \left(\Ib+\Vb^T\Abtil^{-1} \Ub \right)^{-1} \Vb^T \Abtil^{-1}-\Abtil^{-1} \Ub \left(\Ib+\Vb^T{\Ab}^{-1} \Ub \right)^{-1} \Vb^T \Abtil^{-1}\right)}_{\Tb_2}\nonumber\\
       & - \underbrace{\left( \Abtil^{-1} \Ub \left(\Ib+\Vb^T{\Ab}^{-1} \Ub \right)^{-1} \Vb^T \Abtil^{-1} - \Abtil^{-1} \Ub \left(\Ib+\Vb^T{\Ab}^{-1} \Ub \right)^{-1} \Vb^T {\Ab}^{-1}\right)}_{\Tb_3}\nonumber\\
       & - \underbrace{\left( \Abtil^{-1} \Ub \left(\Ib+\Vb^T{\Ab}^{-1} \Ub \right)^{-1} \Vb^T {\Ab}^{-1} - {\Ab}^{-1} \Ub \left(\Ib+\Vb^T{\Ab}^{-1} \Ub \right)^{-1} \Vb^T \Ab^{-1}\right)}_{\Tb_4}.\nonumber
\end{align}
We now bound each of the five terms. First, by our assumptions, $\twonorm{\Eb_1}\leq \epsilon_1$. In the remainder of the proof, we will repeatedly use the fact that $\twonorm{\Abtil^{-1}} \leq \twonorm{\Ab^{-1}}+\epsilon_1$, which follows by the triangle inequality on the two norm of $\Abtil^{-1}=\Ab^{-1}+\Eb_1$. Next,
    \begin{align*}
        \twonorm{\Tb_1} = \twonorm{\Abtil^{-1}\Ub \Eb_2 \Vb^T \Abtil^{-1}} \leq \twonorm{\Abtil^{-1}}^2 \twonorm{\Ub} \twonorm{\Vb} \,\epsilon_2 \leq ( \twonorm{\Ab^{-1}}+\epsilon_1 )^2 \,\lambda\, \epsilon_2.
    \end{align*} 
Using eqn.~(\ref{forward_bound_assumption}), we get
\begin{equation*}
    \twonorm{(\Ib+\Vb^T \Abtil^{-1} \Ub) - (\Ib+\Vb^T\Ab^{-1}\Ub)}=\twonorm{\Vb^T \Eb_1 \Ub} \leq \lambda \epsilon_1 < \frac{1}{2 \alpha}.
\end{equation*}
The last inequality follows by our assumption on $\epsilon_1$ and allows us to apply Lemma~\ref{first_lemma}:
\begin{equation*}
    \twonorm{(\Ib+\Vb^T \Abtil^{-1} \Ub)^{-1} - (\Ib+\Vb^T\Ab^{-1}\Ub)^{-1}} \leq 2\alpha^2 \lambda \epsilon_1.
\end{equation*}
We are now ready to bound the two-norm of $\Tb_2$:
\begin{align}
    \twonorm{\Tb_2}&\leq \twonorm{\Abtil^{-1}}^2 \twonorm{\Ub} \twonorm{\Vb} \twonorm{(\Ib+\Vb^T \Abtil^{-1} \Ub)^{-1} - (\Ib+\Vb^T\Ab^{-1}\Ub)^{-1}}  \nonumber\\
    &\leq \twonorm{\Abtil^{-1}}^2 \twonorm{\Ub} \twonorm{\Vb} \,\, \left( 2 \alpha^2 \lambda \,\epsilon_1 \right)\nonumber \\
    &\leq 2 (\twonorm{\Ab^{-1}} + \epsilon_1)^2 \alpha^2 \lambda^2 \epsilon_1 .\nonumber
\end{align}
Next, we bound the two-norm of $\Tb_3$:
\begin{align*}
    \twonorm{\Tb_3}&=\twonorm{ \Abtil^{-1} \Ub\left( \Ib + \Vb^T \Ab^{-1} \Ub\right)^{-1}  \Vb^T \Eb_1}\\
    &\leq \twonorm{\Abtil^{-1}} \twonorm{\Ub} \twonorm{\left( \Ib + \Vb^T \Ab^{-1} \Ub\right)^{-1}} \twonorm{\Vb}\twonorm{\Eb_1}\\
    &\leq (\twonorm{\Ab^{-1}} + \epsilon_1) \lambda \alpha \epsilon_1.
\end{align*}
Similarly,
\begin{equation*}
\twonorm{\Tb_4} \leq\twonorm{\Ab^{-1}} \,\lambda \alpha \epsilon_1.
\end{equation*}
Finally, we use the triangle inequality to combine all terms and complete the proof of Theorem~\ref{forward_error_thm}.

\subsection{Backward Error Bound}

We first discuss necessary assumptions for our backward error bound; next, we present and discuss our bound and the special case discussed in the introduction; and finally we present its proof.

\subsubsection{Assumptions}\label{sxn:backward:assumptions}
%
%
Recall that $\lambda = \twonorm{\Ub} \twonorm{\Vb}$. For our backward error bound, we will need the following assumptions:
\begin{align}
      & \twonorm{\Ib + \Vb^T \Ab^{-1} \Ub} \leq \beta \label{beta_assumption}\\
      & \Abtil\ \text{and}\ (\Zb^{-1})^{-1} - \Vb^T \Abtil^{-1} \Ub\ \text{are invertible},\label{backward_thm_assumption4}\\
      & \epsilon_1 < \frac{1}{2 \twonorm{\Ab}},  \label{backward_thm_assumption1}\\
      & \epsilon_2 < \frac{1}{2 \left( \beta+ \lambda \, \epsilon_1  \right)} \label{epsilon2_small},\ \text{and} \\
      &  2\left( \beta+ \lambda \, \epsilon_1  \right)^2 \epsilon_2 < \frac{1}{2} \label{epsilon2_small_2}.
\end{align}

\subsubsection{The bound}
    
\begin{theorem}
    Under the assumptions of Section~\ref{sxn:backward:assumptions},
    \begin{align}
    \twonorm{\Bb-\left(\Abtil^{-1} - \Abtil^{-1} \Ub \Zb^{-1} \Vb^T \Abtil^{-1}\right)^{-1}} 
    &\leq 2 \epsilon_1 \twonorm{\Ab}^2 + 4 \lambda\epsilon_2 \left(\beta+ \lambda \, \epsilon_1  \right)^2. \label{backward_fullbound}
\end{align}
\label{backward_thm}\end{theorem}

\begin{corollary}\label{remark:backward}
Assume that $\epsilon_1$, $\epsilon_2$  and $\sigma_{\min}(\Ab)$ are all upper bounded by one and also assume that $\epsilon_1$ and $\epsilon_2$ satisfy the assumptions of Theorem~\ref{backward_thm}. Additionally, assume that we have a small update, i.e., $\lambda \leq .5\cdot \sigma_{\min}(\Ab)$. Then, the right-hand side of the bound of eqn.~(\ref{backward_fullbound}) simplifies to:
\begin{equation*}
    \twonorm{\Bb-\left(\Abtil^{-1} - \Abtil^{-1} \Ub \Zb^{-1} \Vb^T \Abtil^{-1}\right)^{-1}}  \leq 2 \epsilon_1 \twonorm{\Ab}^2 + 8\epsilon_2.
\end{equation*}   
\end{corollary}
\begin{proof}
Following the lines of the proof of Corollary~\ref{remark:forward}, we bound $\twonorm{\Ib + \Vb^T \Ab^{-1}\Ub}$ by a constant:
\begin{align*}
    \twonorm{\Ib + \Vb^T \Ab^{-1}\Ub} &= 1 +  \twonorm{\Vb} \twonorm{\Ab^{-1}} \twonorm{\Ub} \leq 1 + \lambda \twonorm{\Ab^{-1}} \leq 1.5.
\end{align*}
The second term in the right-hand side of eqn.~(\ref{backward_fullbound}) can be bounded as follows:
\begin{align*}
    4 \lambda\epsilon_2 \left( \beta+ \lambda \, \epsilon_1  \right)^2 
   &\leq  4 \lambda\epsilon_2  \left(1.5 +\lambda \epsilon_1 \right)^2\\
   &\leq 8\epsilon_2. 
\end{align*}
The second inequality follows from $\lambda \leq .5 \sigma_{\min}(\Ab) \leq .5$ and $\epsilon_1 \leq 1$.
\end{proof}

\begin{corollary}\label{beta large}
  Assume that $\epsilon_1=\epsilon_2=\epsilon$ and $\sigma_{\min}(\Ab)$ are all upper bounded by one and also assume that $\epsilon_1$ and $\epsilon_2$ satisfy the assumptions of Theorem~\ref{backward_thm}. Additionally, assume that $\beta$ is sufficiently large, i.e., $\beta\geq \max(\lambda,\nicefrac{\twonorm{\Ab}}{\sqrt{\lambda}},1)$. Then, the right-hand side of the bound of eqn.~(\ref{backward_fullbound}) simplifies to:
  \begin{equation*}
    \twonorm{\Bb-\left(\Abtil^{-1} - \Abtil^{-1} \Ub \Zb^{-1} \Vb^T \Abtil^{-1}\right)^{-1}}  \leq 18\lambda \epsilon \beta^2.
\end{equation*} 
\end{corollary}
\begin{proof}
    Using $\beta \geq \nicefrac{\twonorm{\Ab}}{\sqrt{\lambda}}$, we can bound the first term in eqn.~(\ref{backward_fullbound}) as follows:
    \begin{equation*}
        2\epsilon_1 \twonorm{\Ab}^2 \leq 2\lambda\epsilon \beta^2.
    \end{equation*}
    Next, using $\epsilon_1\leq 1$ and $\lambda\leq \beta$, we can bound the second term by
    \begin{align*}
        4\lambda \epsilon_2 (\beta+\lambda \epsilon_1)^2 &\leq 16\lambda \epsilon \beta^2.
    \end{align*}
    Adding the two terms together, we get
    \begin{equation*}
    \twonorm{\Bb-\left(\Abtil^{-1} - \Abtil^{-1} \Ub \Zb^{-1} \Vb^T \Abtil^{-1}\right)^{-1}}  \leq 18 \lambda \epsilon \beta^2.
\end{equation*} 
\end{proof}
We now provide an alternative formulation of Theorem~\ref{backward_thm}.
\begin{corollary}\label{backward old bound}
Under the assumptions of Theorem~\ref{backward_thm},
\begin{equation*}
    \twonorm{\Bb -\left(\Abtil^{-1} - \Abtil^{-1} \Ub \Zb^{-1} \Vb^T \Abtil^{-1}\right)^{-1}} \leq 2\epsilon_1 \twonorm{\Ab}^2 + 4\lambda\epsilon_2 (\kappa(\Vb)\twonorm{\Ab^{-1}\Bb}+\lambda\epsilon_1)^2
\end{equation*}
\end{corollary}
\begin{proof}
    Using eqn.~(\ref{identity1}), we can bound $\beta$ as follows:
    \begin{align*}
        \beta &= \twonorm{\Ib + \Vb^T \Ab^{-1}\Ub}\\
        &=\twonorm{\Vb^T \Ab^{-1}\Bb (\Vb^T)^{\dagger}}\\
        &\leq \twonorm{\Vb^T } \twonorm{ \Ab^{-1}\Bb} \twonorm{ (\Vb^T)^{\dagger}}\\
        &\leq \kappa(\Vb) \twonorm{ \Ab^{-1}\Bb}.
    \end{align*}
    Substituting the above to eqn.~(\ref{backward_fullbound}) gives the result.
\end{proof}
\subsubsection{Proof of Theorem~\ref{backward_thm}}

We now present the proof for our backward error bound.
Recall that $\widetilde{\Bb}^{-1}= \Abtil^{-1} - \Abtil^{-1} \Ub \Zb^{-1} \Vb^T \Abtil^{-1}$.
We first break $\inv{\widetilde{\Bb}^{-1}}-\Bb$ into two parts as follows:
     \begin{align*}
       & \inv{\widetilde{\Bb}^{-1}}-\Bb = \inv{\widetilde{\Bb}^{-1}} - \left( \Ab + \Ub \Vb^T \right)\\
       &=\underbrace{(\Abtil^{-1})^{-1}-\Ab}_{\Sb_1} + \underbrace{\inv{\widetilde{\Bb}^{-1}}-\left((\Abtil^{-1})^{-1}+\Ub \Vb^T\right)}_ {\Sb_2}.
   \end{align*}
To bound $\twonorm{\Sb_1}$, we apply Lemma~\ref{first_lemma}, since eqn.~(\ref{backward_thm_assumption1}) ensures that $\Abtil^{-1}$ is sufficiently close to $\Ab^{-1}$:
\begin{align*}
   \twonorm{\Sb_1}&= \twonorm{(\Abtil^{-1})^{-1}-(\Ab^{-1})^{-1}} \leq 2 \twonorm{\Ab}^2 \epsilon_1. 
\end{align*}
Then, we simplify $\Sb_2$ as follows:
\begin{align*}
   \Sb_2&= \left(\Abtil^{-1} - \Abtil^{-1} \Ub \Zb^{-1} \Vb^T \Abtil^{-1}\right)^{-1}-\left((\Abtil^{-1})^{-1}+\Ub \Vb^T \right)\\
   &= \left[\left(\Ib-\Ub \Zb^{-1} \Vb^T \Abtil^{-1} \right)^{-1} - \left( \Ib + \Ub \Vb^T \Abtil^{-1}\right)  \right] (\Abtil^{-1})^{-1}.
\end{align*}
Next, we apply Woodbury's identity to $\left(\Ib-\Ub \Zb^{-1} \Vb^T \Abtil^{-1} \right)^{-1}$ to get
\begin{align*}
   \Sb_2 &= \left[ \left(  \Ib + \Ub \left((\Zb^{-1})^{-1}-\Vb^T \Abtil^{-1} \Ub\right)^{-1} \Vb^T \Abtil^{-1}\right) - \left( \Ib + \Ub \Vb^T \Abtil^{-1}\right)\right] (\Abtil^{-1})^{-1}\\
    &= \Ub \left[ \left((\Zb^{-1})^{-1}-\Vb^T \Abtil^{-1} \Ub\right)^{-1} -\Ib \right] \Vb^T.
\end{align*}
Therefore, we get
\begin{equation}
    \twonorm{\Sb_2}\leq \lambda \twonorm{\left((\Zb^{-1})^{-1}-\Vb^T \Abtil^{-1} \Ub\right)^{-1} -\Ib}.
    \label{inter_result_1}
\end{equation}
It now remains to bound $\twonorm{\left((\Zb^{-1})^{-1}-\Vb^T \Abtil^{-1} \Ub\right)^{-1} -\Ib}$.
Using eqn.~(\ref{beta_assumption}), we get: 
 \begin{align*}
     \twonorm{\Ib + \Vb^T \Abtil^{-1} \Ub} &\leq \twonorm{\Ib + \Vb^T \Ab^{-1} \Ub} + \twonorm{ \Vb^T \Eb_1 \Ub}\\
     &\leq \beta + \lambda \, \epsilon_1.
 \end{align*}
Using the assumption of eqn.~(\ref{epsilon2_small}), it follows that $\Zb^{-1}$ is close enough to $\inv{\Ib + \Vb^T \Abtil^{-1}\Ub}$. We apply Lemma~\ref{first_lemma} to get:
\begin{equation*}
    \twonorm{\inv{\Zb^{-1}}-\left( \Ib + \Vb^T \Abtil^{-1} \Ub \right)} \leq 2 (\lambda \epsilon_1 + \beta)^2 \epsilon_2.
\end{equation*}
This shows that  $(\Zb^{-1})^{-1} - (\Vb^T \Abtil^{-1} \Ub) $ and $\Ib^{-1}$ are close. Using assumption of eqn.~(\ref{epsilon2_small_2}), we apply Lemma~\ref{first_lemma} again to get:
\begin{align*}
    \twonorm{ \left((\Zb^{-1})^{-1}-\Vb^T \Abtil^{-1} \Ub\right)^{-1} -\Ib} &\leq 2 \times 1^2 \times 2 \left(\beta + \lambda\,\epsilon_1 \right)^2 \epsilon_2\\
    &= 4 \left(\beta + \lambda\,\epsilon_1 \right)^2 \epsilon_2.
\end{align*}
Combining the above bound with eqn.~(\ref{inter_result_1}), we get:
\begin{equation*}
    \twonorm{\Sb_2} \leq 4 \lambda \left(\beta+ \lambda\,\epsilon_1 \right)^2 \epsilon_2.
\end{equation*}
Finally, we use the triangle inequality to conclude the proof. 
\section{Numerical Experiments}\label{sec_numerical_experiment}

The goal of our numerical experiments is to verify the tightness of our forward and backward error bounds as well as examining the dominating terms. We focus on randomly generated matrices and vary the parameters in Theorems~\ref{forward_error_thm} and~\ref{backward_thm} to evaluate our results. The code for reproducing the experiments is publicly available in our GitHub repository\footnote{\url{https://github.com/LinkaiMa/SMW}}.

\subsection{Varying approximation errors}
In the following experiments, we vary the approximation errors $\epsilon_1 \text{ and }\epsilon_2$, while having $\Ab, \Ub, \Vb$ and the update magnitude $\lambda$ fixed. We generate the matrices $\Ab$, $\Ub$, and $\Vb$ with entries sampled independently from the standard Gaussian distribution. The matrix $\Ab$ is $n \times n$ and invertible, while $\Ub$ and $\Vb$ have dimensions $n \times k$. We normalize $\Ub$ and $\Vb$ such that $\twonorm{\Ub} = \twonorm{\Vb} = \sqrt{\lambda}$. This allows us to control the magnitude of the update in our evaluations. 

For the approximate inverses, we generate two standard Gaussian matrices, $\Eb_{1}$ and $\Eb_{2}$, and fix their two-norms to be $\twonorm{\Eb_1} = \epsilon_1$ and $\twonorm{\Eb_2} = \epsilon_2$, respectively. Then, the perturbed inverses are set to $$\Abtil^{-1} = \Ab^{-1} + \Eb_1\quad \mbox{and}\quad \Zb^{-1} = \inv{\Ib + \Vb^T \Abtil^{-1} \Ub} + \Eb_2,$$ corresponding to the assumptions of eqns.~(\ref{eqn:assumption1}) and~(\ref{eqn:assumption2}). With $\Ab$, $\Ub$, $\Vb$, and $\lambda$ fixed, each experiment is repeated 100 times and the average two-norm of the overall error is presented in our plots.

\subsubsection{Forward error bound (Figure~\ref{forward_not_spd})} To examine the dominating term in the forward error bound, we set the errors in the approximate inverses to be equal, i.e., $\epsilon_1 = \epsilon_2 = \epsilon$. Since all the random matrices we used are Gaussian, we expect the capacitance matrix, $\Ib + \Vb^T \Ab^{-1} \Ub$, to be well behaved and only examine the terms in Corollary~\ref{remark:forward}. We investigate values of $\epsilon$ in the interval $[10^{-8},10^2]$. Regarding the magnitude of the update, we examine two settings, i.e., a small update (Figure~\ref{forward_not_spd}, top) and a large update (Figure~\ref{forward_not_spd}, bottom). For the small update, we set $\lambda = .5\sigma_{\min}(\Ab)$ and for the large update, we set $\lambda = .5\sigma_{\max}(\Ab)$. As we can see in Figure~\ref{forward_not_spd}, when the update is small, our forward error bound aligns closely with the simplified bound in eqn.~(\ref{simple_forward_bound}), which accurately captures the behavior of the forward error as $\epsilon$ varies. When the update is large, our forward error bound is less accurate. However, the simplified bound of eqn.~(\ref{simple_forward_bound}) is still close to the actual error.

\begin{figure}[h!]
    \centering
        \includegraphics[width=0.75\linewidth]{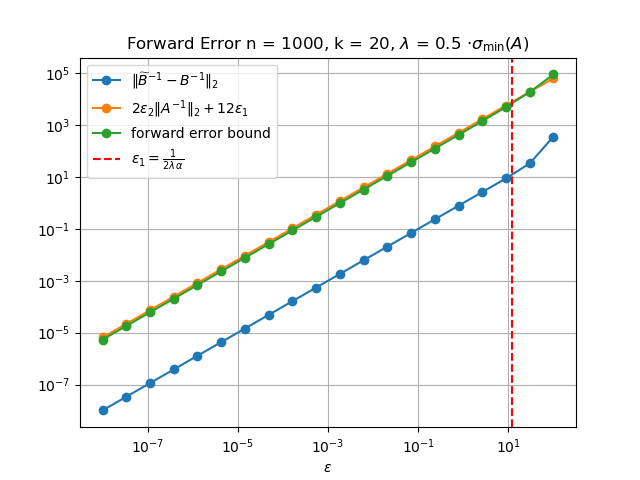}
        \vfill
        \includegraphics[width=0.75\linewidth]{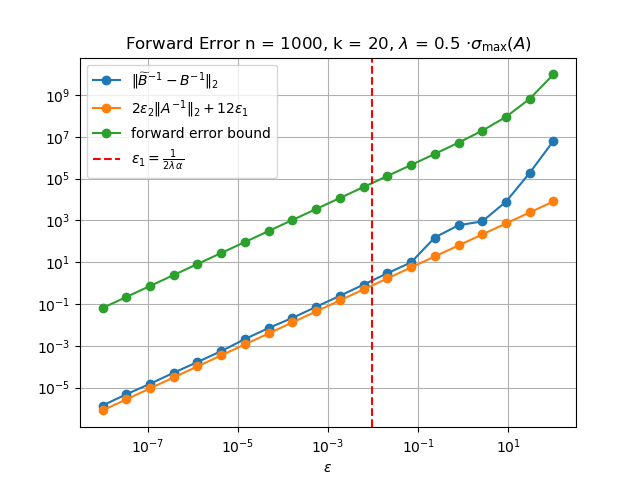}
        \caption{Rank-20 updates of the form $\Bb = \Ab + \Ub \Vb^{T}$, where $\Ab \in \R^{1000 \times 1000}$ is a standard Gaussian matrix.  $\Ub,\Vb \in \R^{1000\times 20}$ are standard Gaussian matrices scaled such that their two-norms (squared) are equal to $\lambda = .5\sigma_{\min}(\Ab)$ (top) or $\lambda = .5 \sigma_{\max}(\Ab)$ (bottom). The horizontal axis represents the values of $\epsilon_{1} = \epsilon_2 = \epsilon$. The \textcolor{blue}{blue dotted line} represents the forward error (i.e. the left-hand side of eqn.~(\ref{forward_bound})), while the \textcolor{green}{green dotted line} shows the respective bound (i.e. the right-hand side of eqn.~(\ref{forward_bound})), and the \textcolor{orange}{orange dotted line} represents the terms in Corollory~\ref{remark:forward}, specifically the term $2\epsilon \twonorm{\Ab^{-1}} + 12 \epsilon$. The  vertical \textcolor{red}{red dashed line} represents the maximum value of $\epsilon$ that satisfies the assumption of eqn.~(\ref{forward_bound_assumption}) for the forward error bound. The top plot shows the small update case, while the bottom plot shows the large update case. }
        \label{forward_not_spd}
\end{figure}

\subsubsection{Backward error bound (Figure~\ref{backward_not_spd})}  To examine the dominating term in the backward error bound, we again assume that $\epsilon_1 = \epsilon_2 = \epsilon$. We evaluate the backward error bound in eqn.~(\ref{backward_fullbound}) as well as our simplified bound in eqn.~(\ref{simple_backward_bound}), across values of $\epsilon$ in the interval $[10^{-8},10^2]$. In addition, we compare this with the backward error from direct inversion with the same inversion error, i.e., assume we can invert $\Bb$ directly with the cost of an error $\epsilon_3 = \epsilon$ represented by a Gaussian matrix. As for the magnitude of update, we examine two settings, i.e., small update (top) and large update (bottom). For the small update, we enforce $\lambda =.5\sigma_{\min}(\Ab)$. For the large update, we enforce $\lambda = .5\sigma_{\max}(\Ab)$. As we can see from Fig~\ref{backward_not_spd}, when the update is small, our backward error bound is essentially the same as the simplified bound in eqn.~(\ref{simple_backward_bound}), which accurately captures the behavior of the backward error when $\epsilon$ varies. Surprisingly, when the update is large, our backward error bound becomes more accurate. The simplified bound still aligns closely with the full bound.

\begin{figure}[h!]
    \centering
        \includegraphics[width=0.75\linewidth]{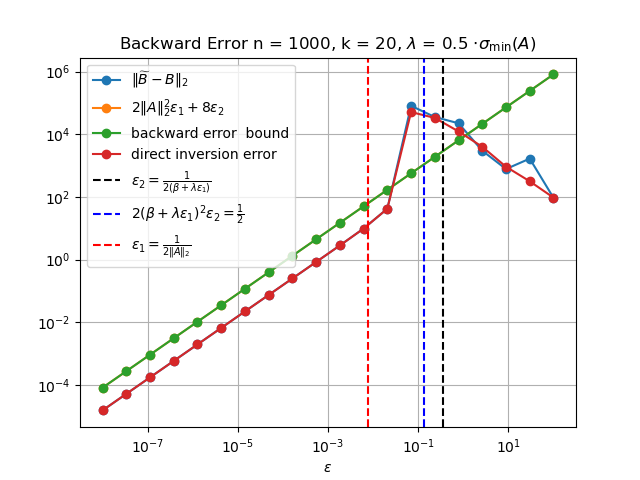}\vfill
        \includegraphics[width=0.75\linewidth]{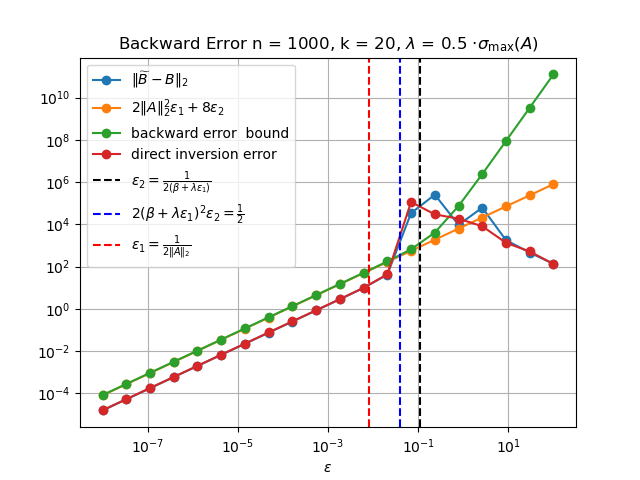}
        \caption{Rank-20 updates of the form $\Bb = \Ab + \Ub \Vb^{T}$, where $\Ab \in \R^{1000 \times 1000}$ is a standard Gaussian matrix.  $\Ub,\Vb \in \R^{1000\times 20}$ are standard Gaussian matrices scaled such that their two-norms (squared) are equal to $\lambda = \frac{1}{2}\sigma_{\min}(\Ab)$ (top) or $\lambda = \frac{1}{2} \sigma_{\max}(\Ab)$ (bottom). The horizontal axis represents the values of $\epsilon_{1} = \epsilon_2 = \epsilon$. The \textcolor{blue}{blue dotted line} represents the backward error (i.e. the left-hand side of eqn.~(\ref{backward_fullbound})), while the \textcolor{green}{green dotted line} shows the respective bound (i.e. the right-hand side of eqn.~(\ref{backward_fullbound})), and the \textcolor{orange}{orange dotted line} represents the terms in Corollary~\ref{remark:backward}, i.e, $2 \twonorm{\Ab}^2 \epsilon_1 + 8\epsilon_2$. The \textcolor{red}{red dotted line} represents the direct inversion backward error. Moreover, the vertical \textcolor{red}{red dashed line}, \textbf{black dashed line}, and \textcolor{blue}{blue dashed line} represent the maximum value of $\epsilon_1 = \epsilon_2 = \epsilon$ that satisfies the assumptions in eqns.~(\ref{backward_thm_assumption1}),~(\ref{epsilon2_small})\,and ~(\ref{epsilon2_small_2}) respectively for the backward error bound. The top plot shows the small update case, while the bottom plot shows the large update case.}
        \label{backward_not_spd}
\end{figure}

\subsection{Varying capacitance matrix}\label{vary capacitance}

In the following experiments, we vary the extreme singular values of the capacitance matrix, i.e., $\alpha = \twonorm{\Ib + \Vb \Ab^{-1}\Ub}$ and $\beta = \twonorm{\inv{\Ib + \Vb \Ab^{-1}\Ub}}$, while having $\Ab$ and $\lambda$ fixed. We manually construct special matrices $\Ab,\Ub,\Vb$ that allow us to control the capacitance matrix. The details of the constructions are technical, hence are deferred to Appendix~\ref{experiment_construction}. Similarly to the previous section, we use Gaussian noises to model the approximate inversion errors. Additionally, for a fixed $\Ab$ and $\lambda$, we repeat the same experiment $100$ times and present the average two-norm error in our plots.

\subsubsection{Forward error bound (Figure~\ref{forward_diagonal})} We vary the value of $\alpha$ over a wide range of values and examine the dominant term in eqn.~(\ref{simple_forward_bound_2}). Regarding the magnitude of update, we examine two settings, i.e., a small update (Figure~\ref{forward_diagonal}, top) and a large update (Figure~\ref{forward_diagonal}, bottom). For the small update, we enforce $\lambda = 2\sigma_{\min}(\Ab)$. For the large update, we enforce $\lambda = 2\sigma_{\max}(\Ab)$. As we can see in Figure~\ref{forward_diagonal}, in both small and large update cases, the dominant term in eqn.~(\ref{simple_forward_bound_2}) aligns closely with our forward error bound (\ref{forward_bound}) and successfully captures the behavior of the error as $\alpha$ grows. We also notice that while our bound stays rather close to the actual error for the small update case (top), it will be way off when the update is large (bottom). This is unfortunately due to the special construction of our experiments (see Appendix~\ref{experiment_construction}).

\begin{figure}[h!]
    \centering
        \includegraphics[width=0.75\linewidth]{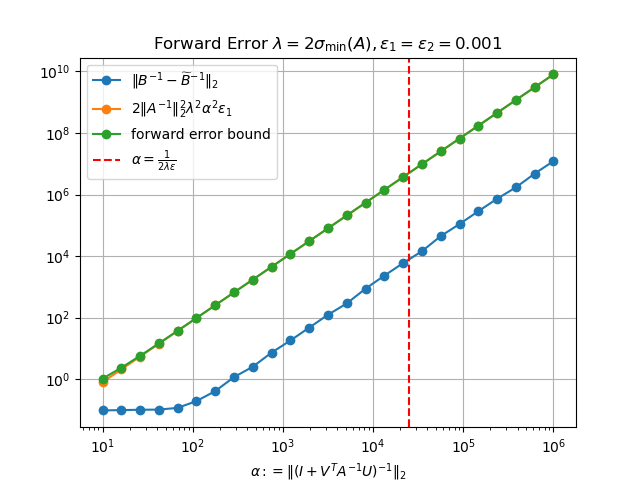}\vfill
        \includegraphics[width=0.75\linewidth]{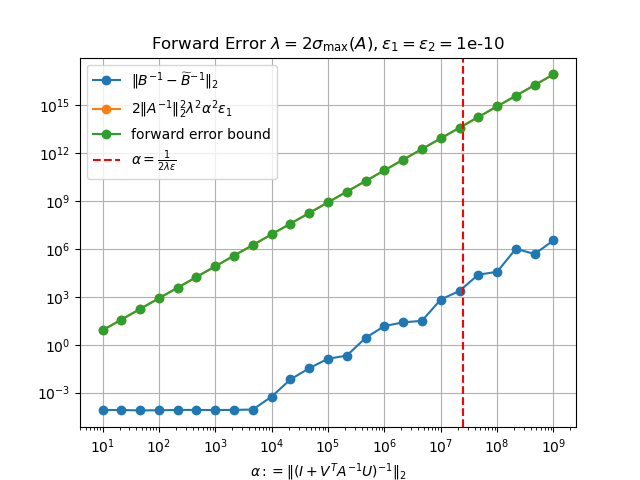}
        \caption{Rank-20 updates of the form $\Bb = \Ab + \Ub \Vb^{T}$, where $\Ab \in \R^{1000 \times 1000},\Ub\in \R^{1000\times 20} \text{ and }\Vb \in \R^{1000\times 20}$ are manually constructed matrices such that $\lambda = 2\sigma_{\min}(\Ab)$ (top) or $\lambda = 2 \sigma_{\max}(\Ab)$ (bottom). We set $\epsilon_1 = \epsilon_2 = 10^{-3}$ for the top plot and $\epsilon_1 = \epsilon_2 = 10^{-10}$ for the bottom plot. The horizontal axis represents the values of $\alpha = \twonorm{\inv{\Ib + \Vb^T \Ab^{-1} \Ub}}$. The \textcolor{blue}{blue dotted line} represents the forward error (i.e. the left-hand side of eqn.~(\ref{forward_bound})), while the \textcolor{green}{green dotted line} shows the respective bound (i.e. the right-hand side of eqn.~(\ref{forward_bound})), and the \textcolor{orange}{orange dotted line} represents the terms in eqn.~(\ref{simple_forward_bound_2}), i.e, $2 \twonorm{\Ab^{-1}}^2 \lambda^2 \alpha^2 \epsilon_1$. Moreover, the vertical \textcolor{red}{red dashed line} represents the maximum value of $\alpha$ that satisfies the assumption in eqn.~(\ref{forward_bound_assumption}) for the forward error bound. The top plot shows the small update case, while the bottom plot shows the large update case.}
        \label{forward_diagonal}
\end{figure}

\subsubsection{Backward error bounds (Figure~\ref{backward_diagonal})} In this experiment, we vary the value of $\beta$ and examine the dominant term in eqn.~(\ref{simple_backward_bound_2}). In terms of the magnitude of the update, we explore two settings: a small update (Figure~\ref{backward_diagonal}, top) and a large update (Figure~\ref{backward_diagonal}, bottom). For the small update, we set $\lambda = 100\sigma_{\min}(\Ab)$\footnote{We need $\lambda = 100\sigma_{\min}(\Ab)$ to guarantee that the values of $\beta$ vary across a wide range. This is due to the fact that if $\lambda\leq c\, \sigma_{\min}(\Ab)$ for some constant $c$, then $\beta =\twonorm{\Ib + \Vb^T \Ab^{-1}\Ub} \leq 1+ \lambda\twonorm{\Ab^{-1}}\leq 1+ c$.}. For the large update, we set $\lambda = 2\sigma_{\max}(\Ab)$. 
As shown in Figure~\ref{backward_diagonal}, in both the small and large update cases the dominant term in eqn.~(\ref{simple_forward_bound_2}) essentially coincides with our backward error bound in eqn.~(\ref{backward_fullbound}) when $\beta$ is large. Moreover, our bound closely approximates the actual error, differing only by a small constant.

\begin{figure}[h!]
    \centering
        \includegraphics[width=0.75\linewidth]{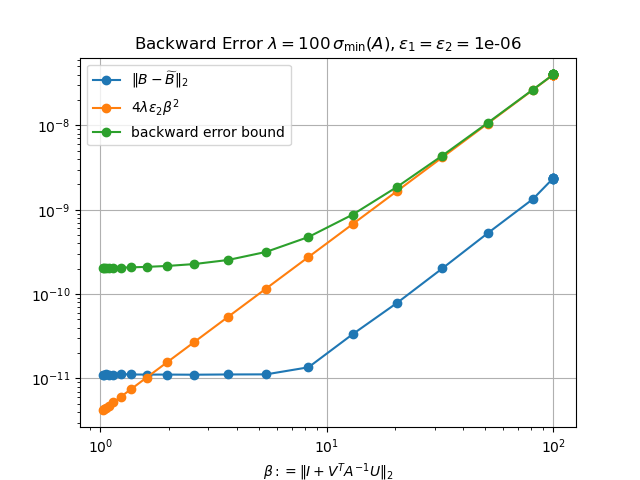}\vfill
        \includegraphics[width=0.75\linewidth]{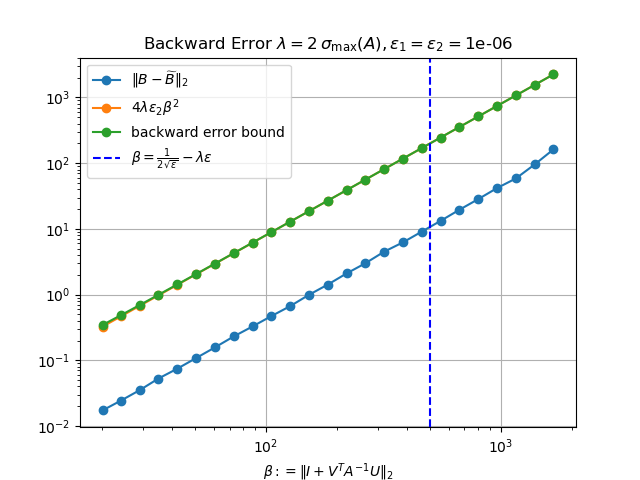}
        \caption{Rank-20 updates of the form $\Bb = \Ab + \Ub \Vb^{T}$, where $\Ab \in \R^{1000 \times 1000},\Ub\in \R^{1000\times 20} \text{ and }\Vb \in \R^{1000\times 20}$ are manually constructed matrices such that $\lambda = 100\sigma_{\min}(\Ab)$ (top) or $\lambda = 2 \sigma_{\max}(\Ab)$ (bottom). We set $\epsilon_1 = \epsilon_2 = 10^{-6}$ for all the experiments. The horizontal axis represents the values of $\beta = \twonorm{{\Ib + \Vb^T \Ab^{-1} \Ub}}$. The \textcolor{blue}{blue dotted line} represents the backward error (i.e. the left-hand side of eqn.~(\ref{backward_fullbound})), while the \textcolor{green}{green dotted line} shows the respective bound (i.e. the right-hand side of eqn.~(\ref{backward_fullbound})), and the \textcolor{orange}{orange dotted line} represents the terms in eqn.~(\ref{simple_forward_bound_2}), i.e, $4\lambda \epsilon_2 \beta^2$. The top plot shows the small update case, while the bottom plot shows the large update case. For the top plot, we checked that all the 3 assumptions from Thm~\ref{backward_thm}: eqns.~(\ref{backward_thm_assumption1}), ~(\ref{epsilon2_small}) and (\ref{epsilon2_small_2})\,are satisfied in all the experiments. For the bottom plot, we checked that both eqn.~(\ref{backward_thm_assumption1}) and eqn.~(\ref{epsilon2_small}) are satisfied in all the experiments. The assumption eqn.~(\ref{epsilon2_small_2}) will be violated when $\beta$ becomes too large. The \textcolor{blue}{blue dashed line} represents the maximum value of $\beta$ such that eqn.~(\ref{epsilon2_small_2}) is satisfied.}
        \label{backward_diagonal}
\end{figure}
\section{Future Work}\label{sec_future}
Several promising directions emerge from our investigation into the numerical stability of the Sherman-Morrison-Woodbury (SMW) formula. First, our experimental results (Section~\ref{sec_numerical_experiment}) reveal an interesting phenomenon where the simplified bounds in both forward (Theorem~\ref{forward_error_thm}) and backward (Theorem~\ref{backward_thm}) error bounds more closely approximate the actual error when the update matrices $(\Ub,\Vb)$ have large norms. While Theorems~\ref{forward_error_thm} and~\ref{backward_thm} provide general error bounds, they do not fully explain this observed correlation between update magnitude and bound tightness. A theoretical analysis of this relationship could lead to sharper error bounds in cases of large updates, potentially through the identification of problem-specific conditions that enable tighter approximations of the leading error terms.

Second, the empirical validity of our forward error bounds under relaxed conditions (Figure~\ref{forward_not_spd}) suggests that some technical assumptions in our analysis, particularly the assumption of eqn.~(\ref{forward_bound_assumption}), could be weakened or removed using alternative proof strategies. This invites further investigation into weaker sufficient conditions for stability, possibly through the use of alternative matrix inverse perturbation arguments.

Finally, in optimization algorithms, including interior point methods (IPMs) for linear programming and multiplicative weights update (MWU) for $\ell_p$ norm minimization, the SMW formula is applied multiple times throughout the algorithm, often to a projection matrix \cite{chowdhuri2020speeding,dexter2022convergence,chowdhury2022faster,ghadiri2023bit,ghadiri2024improving,anand2024bit}. Moreover, some works apply the SMW formula for solving structured linear regression problems \cite{fahrbach2022subquadratic}. An interesting direction for future research is to investigate whether our bounds could be improved in these special cases.

\clearpage

\printbibliography

\clearpage

\noindent\appendix 
\section{Proving the backward error bound of eqn.~(\ref{Ghadiri_main_result_norm_2})}\label{Ghadiri_main_result_norm_2_proof}
Let $\Ab$ and $\tilde{\Ab}$, both in $\R^{n \times n}$, be invertible matrices. Moreover, let $\Ub, \Vb \in \R^{n \times k}$ be such that $\Ab + \Ub \Vb^{T}$ is invertible. Let $\rho \geq 1$ and $\gamma > 0$ such that 
$\max\{\twonorm{\Ub},\twonorm{\Vb}\} \leq \gamma$ and $\max\{\twonorm{\Ab}, \, \twonorm{\Ab^{-1}}, \, \twonorm{\Ab + \Ub \Vb^{T}}, \, \twonorm{(\Ab + \Ub \Vb^{T})^{-1}}\} \leq \rho$.
Additionally, suppose
\begin{align} \label{app_ass_1}
    \twonorm{\tilde{\Ab} - \Ab} \leq \epsilon_{1} < 1.
\end{align}
Let $\Zb \in \R^{k \times k}$ be an invertible matrix and assume that
\begin{align} \label{app_ass_2}
    \twonorm{\Zb^{-1} - (\Ib + \Vb^{T}\tilde{\Ab}^{-1} \Ub)^{-1}} \leq \epsilon_{2}.
\end{align}
By the triangle inequality and eqn.~(\ref{app_ass_1}), we have
\begin{align}
    & \left\|\left( \tilde{\Ab}^{-1} - \tilde{\Ab}^{-1} \Ub \Zb^{-1} \Vb^{T} \tilde{\Ab}^{-1} \right)^{-1} - (\Ab + \Ub \Vb^{T}) \right\|_2 \notag \\
    &\leq \left\|\left( \tilde{\Ab}^{-1} - \tilde{\Ab}^{-1} \Ub \Zb^{-1} \Vb^{T} \tilde{\Ab}^{-1} \right)^{-1} - (\tilde{\Ab} + \Ub \Vb^{T}) \right\|_2 + \left\| \tilde{\Ab} - \Ab \right\|_2  \notag \\
    &\leq \left\|\left( \tilde{\Ab}^{-1} - \tilde{\Ab}^{-1} \Ub \Zb^{-1} \Vb^{T} \tilde{\Ab}^{-1} \right)^{-1} - (\tilde{\Ab} + \Ub \Vb^{T}) \right\|_2 + \epsilon_{1}. \label{app_trian}
\end{align}
Note that $\Sb = \Ib + \Vb^{T} \tilde{\Ab}^{-1} \Ub$ is a Schur complement of the following matrix
\begin{align*}
    \Tb = \begin{pmatrix} \Ib & \Vb^{T} \\ \Ub & -\tilde{\Ab} \end{pmatrix} 
    = \begin{pmatrix} \Ib & \boldsymbol{0} \\ \Ub & \Ib \end{pmatrix} \begin{pmatrix} \Ib & \boldsymbol{0} \\ \boldsymbol{0} & -\tilde{\Ab} - \Ub \Vb^{T} \end{pmatrix} \begin{pmatrix} \Ib & \Vb^{T} \\ \boldsymbol{0} & \Ib \end{pmatrix}.
\end{align*}
Since $\Ib$ and $-\tilde{\Ab} - \Ub \Vb^{T}$ (the Schur complement) are invertible matrices, $\Tb$ is also invertible and
\begin{align*}
    \Tb^{-1} = \begin{pmatrix}
        \Ib & -\Vb^{T} \\
        \boldsymbol{0} & \Ib
    \end{pmatrix} \begin{pmatrix} \Ib & \boldsymbol{0} \\ \boldsymbol{0} & -\tilde{\Ab} - \Ub \Vb^{T} \end{pmatrix}^{-1}  \begin{pmatrix} \Ib & \boldsymbol{0} \\ -\Ub & \Ib \end{pmatrix}.
\end{align*}
Notice that since $\Sb$ is the Schur complement of $\Tb$, $\Sb^{-1}$ is a principal submatrix of $\Tb^{-1}$. Therefore,
\begin{align}
    \left\|  \Sb^{-1} \right\|_2 &\leq \left\| \Tb^{-1}  \right\|_2 \notag
    \leq \left( 1+ \twonorm{\Ub} \right)  \left( 1+ \twonorm{\Vb} \right) \max \{ 1,  \left\|  \left( \tilde{\Ab} + \Ub \Vb^{T} \right)^{-1} \right\|_2  \}  \notag \\
    &\leq (1+\gamma)^2 \rho \label{app_schur}.
\end{align}
Moreover,
\begin{align*}
    \twonorm{\Sb} \leq 1 + \left\|  \Vb^{T} \right\|_2 \left\|  \tilde{\Ab}^{-1} \right\|_2 \left\|  \Ub \right\|_2 
    \leq \gamma^{2} \rho + 1. 
\end{align*}
Using the assumption of eqn.~(\ref{app_ass_2}) and Lemma~\ref{first_lemma} for $\epsilon_{2} \leq \nicefrac{1}{2(1+\gamma^2 \rho)}$, we get
\begin{align} \label{app_use_first_lemma}
    \left\|  \Zb - \left( \Ib + \Vb \tilde{\Ab}^{-1} \Ub \right) \right\|_2 \leq 2 \left( \gamma^{2} \rho + 1 \right)^2 \epsilon_{2}.
\end{align}
Now define
\begin{align*}
    \Mb = \begin{pmatrix}
        \Zb - \Vb^{T} \tilde{\Ab}^{-1} \Ub & \Vb^{T} \\ \Ub & -\tilde{\Ab}
    \end{pmatrix}.
\end{align*}
Then by eqn.~(\ref{app_use_first_lemma}),
\begin{align} \label{app_result_1_first_lemma}
    \twonorm{\Mb - \Tb} \leq 2 \left( \gamma^2 \rho +1 \right)^{2} \epsilon_{2}.
\end{align}
Note that $\Mb$ is invertible by construction and thus we can use Lemma~\ref{first_lemma} with eqn.~(\ref{app_use_first_lemma}) and the fact that $2 \left( \gamma^{2} \rho + 1 \right)^2 \epsilon_{2} \leq \nicefrac{1}{2 \left(1 + \gamma \right)^2 \rho} $ to get
\begin{align*}
     \left\|  \Mb^{-1} - \Tb^{-1} \right\|_2 \leq 4 (1 + \gamma)^{4} \rho ^{2} \left( \gamma^{2} \rho +1 \right)^{2} \epsilon_{2}.
\end{align*}
Observing that $-\tilde{\Ab} - \Ub \Vb^{T}$ and $-\tilde{\Ab} - \Ub \left( \Zb - \Vb^{T} \tilde{\Ab}^{-1} \Ub \right)^{-1} \Vb^{T}$ are the Schur complements of the corresponding blocks of $\Tb$ and $\Mb$ (respectively), we can derive
\begin{align*}
   & \left\| \left( \tilde{\Ab} + \Ub \left( \Zb - \Vb^{T} \tilde{\Ab}^{-1} \Ub \right)^{-1} \Vb^{T}\right)^{-1} - \left(\tilde{\Ab} + \Ub \Vb^{T}  \right)^{-1} \right\|_2 \\ &\leq  \left\|  \Mb^{-1} - \Tb^{-1} \right\|_2 \leq 4 (1 + \gamma)^{4} \rho ^{2} \left( \gamma^{2} \rho +1 \right)^{2} \epsilon_{2}.
\end{align*}
Using the Woodbury identity, we get
\begin{align*}
    & \left( \tilde{\Ab} + \Ub \left( \Zb - \Vb^{T} \tilde{\Ab}^{-1} \Ub \right)^{-1} \Vb^{T}\right)^{-1} \\ &= \tilde{\Ab}^{-1} - \tilde{\Ab}^{-1} \Ub \left(\Zb - \Vb^{T} \tilde{\Ab}^{-1} \Ub + \Vb^{T} \tilde{\Ab}^{-1} \Ub \right)^{-1} \Vb^{T} \tilde{\Ab}^{-1} \\
    &= \tilde{\Ab}^{-1} - \tilde{\Ab}^{-1} \Ub \Zb^{-1} \Vb^{T} \tilde{\Ab}^{-1}. 
\end{align*}
It now follows that
\begin{align*}
   \left\|  \left(  \tilde{\Ab}^{-1} - \tilde{\Ab}^{-1} \Ub \Zb^{-1} \Vb^{T} \tilde{\Ab}^{-1} \right) - \left(\tilde{\Ab} + \Ub \Vb^{T}  \right)^{-1} \right\|_2 \leq 4 (1 + \gamma)^{4} \rho ^{2} \left( \gamma^{2} \rho +1 \right)^{2} \epsilon_{2}.
\end{align*}
Thus, using $\left\| \left( \Ab + \Ub \Vb^{T} \right)^{-1}\right\|_2 \leq \rho$, assuming $4 (1 + \gamma)^{4} \rho ^{2} \left( \gamma^{2} \rho +1 \right)^{2} \epsilon_{2} \leq \nicefrac{1}{2 \rho}$, and using Lemma~\ref{first_lemma}, we get
\begin{align*}
    \left\|  \left(  \tilde{\Ab}^{-1} - \tilde{\Ab}^{-1} \Ub \Zb^{-1} \Vb^{T} \tilde{\Ab}^{-1} \right)^{-1} - \left(\tilde{\Ab} + \Ub \Vb^{T}  \right) \right\|_2 \leq 8 (1 + \gamma)^{4} \rho^{4} \left( \gamma^{2} \rho +1  \right)^{2} \epsilon_{2}.
\end{align*}
The final bound follows after combining the above bound with eqn.~(\ref{app_trian}). Recall that $\rho \geq 1$ and assume that $\gamma \leq \rho$ to conclude
\begin{align*}
     \left\|\left( \tilde{\Ab}^{-1} - \tilde{\Ab}^{-1} \Ub \Zb^{-1} \Vb^{T} \tilde{\Ab}^{-1} \right)^{-1} - (\Ab + \Ub \Vb^{T}) \right\|_2 \leq 512\epsilon_2\rho^{14} + \epsilon_{1}.
\end{align*}


\section{Proof of Lemma~\ref{first_lemma}}\label{first_lemma_proof}
    Let $\rho > 1$ such that $\| \Nb \|_{2}, \, \| \Nb \|_{2}^{-1} \leq \rho$. Additionally assume that $\| \Mb - \Nb \|_{2} \leq \epsilon \leq \frac{1}{2 \rho}$, and let $\Eb = \Mb - \Nb$. By Woodbury identity:
    \begin{equation}
        \Mb^{-1} = \Nb^{-1} - \Nb^{-1}\Eb (\Ib + \Nb^{-1} \Eb)^{-1} \Nb^{-1} \nonumber
    \end{equation}
    Also note that:
    \begin{equation}
        \left( \Ib + \Nb^{-1} \Eb \right)^{-1} = \left( \Ib + \Nb^{-1} \Mb - \Ib \right)^{-1} = \Mb^{-1} \Nb. \nonumber
    \end{equation}
    Therefore
    \begin{align*}
        \|\Mb^{-1}\|_2 &= \| \Nb^{-1} - \Nb^{-1} \Eb \Mb^{-1}\|_2 \\&\leq \|\Nb^{-1}\|_2 + \|\Nb^{-1} \Eb \Mb^{-1}\|_2\\
        &\leq \|\Nb^{-1}\|_2 + \|\Nb^{-1}\|_2 \|\Eb\|_2 \|\Mb^{-1}\|_2\\
        &\leq \rho + \rho \,\epsilon \,\|\Mb^{-1}\|_2.
    \end{align*}
    Therefore 
    \begin{equation*}
        \|\Mb^{-1}\|_2 \leq \frac{\rho}{1 - \rho \varepsilon} \leq 2\rho,
    \end{equation*}
    which implies that
    \begin{align*}
        \|\Mb^{-1} - \Nb^{-1}\|_2 = \|\Nb^{-1} \Eb \Mb^{-1}\|_2
        \leq \|\Nb^{-1}\|_2 \|\Eb\|_2 \|\Mb^{-1}\|_2  \leq 2\rho^2  \epsilon.
    \end{align*}

\section{Proof of Lemma~\ref{lem:lemma_NL_1}}\label{appendix_proof_1}
From our assumptions, $\Ub$ and $\Vb$ have full column rank. Therefore, $\Ub^{\dagger} \Ub = \Ib$, $\Vb^{\dagger} \Vb = \Ib$, and $\Vb^T \Pinv{\Vb^T} = \Ib$.
To prove eqn.~(\ref{identity1}), we start with
\begin{align*}
    \Vb^T \Ab^{-1} \Bb (\Vb^T)^{\dagger} &=  \Vb^T \Ab^{-1} (\Ab + \Ub \Vb^T) (\Vb^T)^{\dagger}\\
    &= \Vb^T \left( \Ib +\Ab^{-1} \Ub \Vb^T \right)(\Vb^T)^{\dagger}\\
    &= \Vb^T (\Vb^T)^{\dagger} + \Vb^T \Ab^{-1} \Ub \Vb^T (\Vb^T)^{\dagger}\\
    &= \Ib + \Vb^T \Ab^{-1} \Ub.
\end{align*}
Next, to establish eqn.~(\ref{identity2}), we first simplify $\Vb^T \Bb^{-1} \Ab (\Vb^T)^{\dagger}$ as follows:
\begin{align*}
    \Vb^T \Bb^{-1} \Ab (\Vb^T)^{\dagger} &= \Vb^T \Bb^{-1} \left( \Bb - \Ub \Vb^T \right)(\Vb^T)^{\dagger}\\
    &= \Vb^T \left(\Ib - \Bb^{-1} \Ub \Vb ^T \right) (\Vb^T)^{\dagger}\\
    &= \Vb^T (\Vb^T)^{\dagger} - \Vb^T \Bb^{-1} \Ub \Vb^T (\Vb^T)^{\dagger}\\
    &= \Ib - \Vb^T \Bb^{-1} \Ub.
\end{align*}
Then, we prove eqn.~(\ref{identity2}) by verifying $\left( \Ib + \Vb^T \Ab^{-1} \Ub \right)^{-1}\left( \Ib + \Vb^T \Ab^{-1} \Ub \right)=\Ib$ (and vice-versa, omitted):
\begin{align*}
     \left( \Ib + \Vb^T \Ab^{-1} \Ub \right)^{-1}\left( \Ib + \Vb^T \Ab^{-1} \Ub \right) &= \Vb^T \Bb^{-1} \Ab (\Vb^T)^{\dagger} \left( \Ib + \Vb^T \Ab^{-1} \Ub\right) \\&= \left(  \Ib - \Vb^T \Bb^{-1} \Ub\right)  \left( \Ib + \Vb^T \Ab^{-1} \Ub\right)\\
     &= \Ib + \Vb^T \Ab^{-1} \Ub - \Vb^T \Bb^{-1} \Ub - \Vb^T \Bb^{-1} \Ub \Vb^T \Ab^{-1} \Ub\\
     &= \Ib + \Vb^T(\Ab^{-1}-\Bb^{-1})\Ub - \Vb^T \Bb^{-1} (\Bb-\Ab)\Ab^{-1}\Ub\\
     &= \Ib  + \Vb^T(\Ab^{-1}-\Bb^{-1})\Ub - \Vb^T (\Ab^{-1}-\Bb^{-1}) \Ub\\
     &= \Ib.
\end{align*}

\section{Details on the experiments in Section~\ref{vary capacitance}}\label{experiment_construction}

\subsection{Matrix constructions for forward error experiments}

We first generate random Gaussian matrices and perform a QR decomposition to obtain $n \times n$ orthogonal matrices $\Ub_{\Ab}$ and $\Vb_{\Ab}$. Then, we set the singular values of $\Ab$ to 
$$\bold{S}_{\Ab} = \texttt{np.logspace(2,-2,n)}.$$ 
Let $\bold{\Sigma}_{\Ab} = \text{diag}(\Sb_{\Ab})$, then, we construct $\Ab$ as $\Ab = \Ub_{\Ab} \bold{\Sigma}_{\Ab} \Vb_{\Ab}^T$. Next, we generate $\Qb \in \mathbb{R}^{n\times k}$ such that 
$$ \Qb^T = \begin{bmatrix} 0;\ \Ib_k\end{bmatrix}.$$ 
Finally, we construct a diagonal matrix $\Sb \in \mathbb{R}^{n \times n}$ that enables us to control the parameters $\alpha$ and $\lambda$. Specifically, the top $n-k$ entries on the diagonal of $\Sb$ are set to zero and the $(n-k+1)$-st entry is set to $\lambda$. The last entry on the diagonal is set to $(\nicefrac{1}{\alpha}-1)\sigma_{\min}(\Ab)$. For $j=n-k+1, \ldots, n-1$, we set $\Sb[j,j] = |\nicefrac{1}{\alpha}-1|\sigma_{j}(\Ab)$. The final construction is 
$$\Ub = \Ub_{\Ab}\Qb\quad \text{and}\quad \Vb = \Vb_{\Ab}\Sb\Qb.$$
Using this construction, we can check that the smallest singular value of the capacitance matrix is indeed $\alpha$:
\begin{align*}
    \Ib + \Vb^T \Ab^{-1}\Ub &= \Ib + \Qb^T \Sb^T \Vb_{\Ab}^T \Vb_{\Ab} \bold{\Sigma}_{\Ab}^{-1} \Ub_{\Ab}^T \Ub_{\Ab} \Qb\\
    &=\Ib + \Qb^T \Sb^T \bold{\Sigma}_{\Ab}^{-1} \Qb.
\end{align*}
Let $\lfloor \Ab \rfloor_k$ denote the bottom-right $k\times k$ block of a $n \times n$ matrix. Then
\begin{align*}
    \Ib + \Vb^T \Ab^{-1}\Ub &= \Ib + \lfloor \Sb^T \bold{\Sigma}_{\Ab}^{-1} \rfloor_k.
\end{align*}
By our construction, 
\begin{align*}
   \Ib + \lfloor \Sb^T \bold{\Sigma}_{\Ab}^{-1} \rfloor_k= \begin{bmatrix}
1+\frac{\lambda}{\sigma_{n-k}(A)} & 0 & \cdots & 0 \\
0 &1+ \left|\frac{1}{\alpha}-1\right| & \ddots & \vdots \\
\vdots & \ddots &1+ \left|\frac{1}{\alpha}-1\right| & 0 \\
0 & \cdots & 0 & \frac{1}{\alpha}
\end{bmatrix},
\end{align*}
which has its smallest singular value equal to $\nicefrac{1}{\alpha}$.

As a final note, recall that in the proof of Theorem~\ref{forward_error_thm}, we used the inequality $\|\Vb^T \Eb_2 \Ub\|_2 \leq \lambda \epsilon_2$. This bound is tight when the singular vectors corresponding to the largest singular values of $\Vb$, $\Eb_2$, and $\Ub$ roughly align. When the update is large, i.e. $\lambda \approx \sigma_{\max}(\Ab)$, then the term $\sigma_{\max}(\Vb) (= \lambda)$ becomes significantly larger than the other singular values of $\Vb$. However, the top singular vectors of $\Vb$ do not align with those of $\Eb_2$ and $\Ub$ and, as a result, our bound is pessimistic.

\subsection{Matrix constructions for backward error experiments} For the small update setting, we generate a diagonal matrix $\Ab$ as follows: The top $60\%$ of the diagonal entries are generated using
$$\texttt{np.logspace}(-2,-8,\texttt{int}(0.6\cdot n)).$$ 
The remaining entries on the diagonal are set to be $10^{-8}$. For the large update, the diagonal entries are generated using
$$\texttt{np.logspace(2,-2,n)}.$$ 
For each experiment, we first create an orthogonal matrix $\Qb\in \mathbb{R}^{n\times k}$ of the form 
$$\Qb^T = \begin{bmatrix} 0;\ \Ib_k;\ 0\end{bmatrix},$$ 
where the position of the first row of $\Ib_k$ increases from around $\nicefrac{n}{4}$ to $\nicefrac{3n}{4}$, in increments of $k$. For each $\Qb$, we set $\Ub = \Ab \Qb$. Then we generate a standard Gaussian matrix $\Sb$ and normalize it to satisfy $\twonorm{\Sb} = \nicefrac{\lambda}{\twonorm{\Ub}}$. Finally, we set $\Vb = \Qb \Sb^T$. 
To better understand this construction, we can compute the resulting capacitance matrix as follows:
\begin{align*}
    \Ib + \Vb^T \Ab^{-1} \Ub = \Ib + \Sb \Qb^T \Ab^{-1} \Ab \Qb
    = \Ib + \Sb.
\end{align*}
By varying $\Qb$, we can vary $\twonorm{\Ub}$ and therefore $\twonorm{\Sb}$, which directly changes $\beta = \twonorm{\Ib + \Vb^T \Ab^{-1} \Ub} = \twonorm{\Ib + \Sb}$.

\end{document}